\documentclass[sn-vancouver,Numbered]{sn-jnl}

\pdfoutput=1
\usepackage{graphicx}%
\usepackage{multirow}%
\usepackage{amsmath,amssymb,amsfonts}%
\usepackage{amsthm}%
\usepackage{mathrsfs}%
\usepackage[title]{appendix}%
\usepackage{xcolor}%
\usepackage{textcomp}%
\usepackage{manyfoot}%
\usepackage{booktabs}%
\usepackage{algorithm}%
\usepackage{algorithmicx}%
\usepackage{algpseudocode}%
\usepackage{listings}%

\newtheorem{thm}{Theorem}[section]
\newtheorem{lem}[thm]{Lemma}

\newtheorem{rem}[thm]{Remark}
\newtheorem{pro}[thm]{Proposition}

\newtheorem*{mthma}{Main Theorem A}
\newtheorem*{mthmb}{Main Theorem B}

\begin{document}

\title[Ramification and unicity theorems for Gauss maps of complete space-like stationary surfaces in four-dimensional Lorentz-Minkowski space]{Ramification and unicity theorems for Gauss maps of complete space-like stationary surfaces in four-dimensional Lorentz-Minkowski space} 


\author{\fnm{Li} \sur{Ou}}\email{lou19@fudan.edu.cn}




\affil{\orgdiv{School of Mathematics}, \orgname{Sichuan university}, \orgaddress{\street{No.24, Renmin South Road}, \city{Chengdu}, \postcode{610064}, \state{Sichuan Province}, \country{China}}}




\abstract{ 
In this paper, we  investigate the value distribution properties for Gauss maps of space-like stationary surfaces in four-dimensional Lorentz-Minkowski space $\mathbb{R}^{3,1}$, focusing on aspects such as the number of totally ramified points and unicity properties. We not only obtain general conclusions similar to situations in four-dimensional Euclidean space, but also consider the space-like stationary surfaces with rational graphic Gauss image, which is an extension of degenerate space-like stationary surfaces. } 





\maketitle

\section{Introduction}\label{sec1}
It is well known that value distribution properties for Gauss maps of regular minimal surfaces  in $\mathbb R^n$ play  a crucial  role in the theory of minimal surfaces. For a minimal surface $M$ in $\mathbb R^3$, the Gauss map is defined by the unit normal vector $G(p)\in \mathbb{S}^2$, for $p\in M$.  The surface $M$ is canonically considered as an open Riemann surface with some conformally metric $ds^2$ and by the minimality of $M$, the Gauss map is a meromorphic function under the sphere stereographic projection. S. S. Chern \cite{MR0180926} introduced the generalized Gauss map $G$ for an oriented regular surface in  $\mathbb R^n$, which map $p\in M$ to the point in 
$Q_{n-2}=\{[(z_1,\cdots, z_n)]\in \mathbb {CP}^{n-1}: z_1^2+\cdots +z_n^2=0\}$ corresponding to the oriented tangent plane of $M$ at $p$.
Then $G$ is a holomorphic map from the Riemann surface $M$ to $\mathbb {CP}^{n-1}$. So there are many analogous results  between Gauss maps of minimal surfaces and meromorphic mappings.

One of them is the Picard theorem. Through the efforts of R. Osserman \cite{Om-pf, R.O-large, R.O-Global}, F. Xavier \cite{F.Xavier}, Mo-Osserman \cite{Mo-Osserman}, H. Fujimoto \cite{H.F1} finally proved that  the  Gauss map of a non-flat complete  minimal surface in $\mathbb R ^3$ can omit at most $4$ points in $\mathbb S^2$.  H. Fujimoto \cite{H.F1} also gave the estimate for the number of exceptional values for  Gauss map of  minimal surfaces in $\mathbb R ^4$.  The number "four" in $\mathbb R ^3$ is the best possible upper bound since a lot of examples of complete minimal surfaces whose Gauss maps miss $4$ points exist \cite{R.O-large, R.O-Survey}. And the geometry interpretation of the maximal number exceptional value is given by Y. Kawakami \cite{MR3078266} (for $\mathbb R ^3$) and R. Aiyama,
K. Akutagawa, S. Imagawa and Y. Kawakami \cite{yama} (for $\mathbb R ^4$ ). The number of exceptional values of Gauss map of minimal surfaces in $\mathbb R^n$ were investigated by Fujimoto \cite{MR1037406} and M. Ru \cite{MR1131437}. 
 
The second one is the ramification problem. In 1992, H. Fujimoto \cite{MR1167375} studied the ramification of  Gauss maps of complete regular minimal surfaces in $\mathbb R^3$. One says that $g: M\to \bar{\mathbb C}=C\cup\{\mathbb  \infty\}$ is ramified over a point $a_j \in \bar{\mathbb C}(1\le j\le q)$ with multiplicity at least $e_j(1\le j\le q)$ if all zeros of the function $g(z)-a_j$ have orders at least $e_j$. If the image of $g$ omits $a_j$, one will say that $g$ is ramified over $a_j$ with multiplicity $e_j=\infty$. 
Specifically, H. Fujimoto obtained the following result:
 \begin{thm} \cite{MR1167375}
 \label{r3}
  Let M be a non-flat complete minimal surface in $\mathbb R^3$. If there are $q(q>4)$ distinct points  $a_1,\cdots, a_q\in \bar{\mathbb C}$  such that the Gauss map of M is ramified over $a_j$ with multiplicity at least $e_j$ for each $j$, then 
  $$\sum_{j=1}^q(1-1/e_j)\le 4.$$
 \end{thm}
In particular, if the Gauss map omits five distinct points, then $M$ must be flat. 
In 1993, M. Ru \cite{MR1191614} extended the ramification result to Gauss maps of minimal surfaces in $\mathbb R^n$. Afterwards, J. Lu \cite{MR2272139}, S. J. Kao \cite{MR1129370}, P. H. Ha, L. B. Phuong, P. D. Thoan, G. Dethloff \cite{MR3270425, MR3442592} studied the number of  exceptional values and ramification of the Gauss map of complete minimal surfaces in $\mathbb R^n$ on annular end. 

The third one is the unicity problem.  H. Fujimoto \cite{MR1219880} also investigated the uniqueness theorem for Gauss maps of minimal surfaces in $\mathbb R^n$, which is analogue to the Nevanlinna unicity theorem \cite{MR1555233} for meromorphic functions on the complex plane $\mathbb C$: Two meromorphic functions on $\mathbb C$ sharing $5$ distinct values must be identically equal to each other. Here, two functions $g_1, g_2$ share value $a$ means $g_1^{-1}(a)=g_2^{-1}(a)$. For the Gauss map of minimal surface,  H. Fujimoto proved the following theorem:
\begin{thm}\cite{MR1219880}\label{uni}
    Let $M$ and $\hat{M}$ be two non-flat minimal surfaces in $\mathbb R^3$ with their Gauss maps $g$ and $\hat{g}$ respectively. Suppose that there is a conformal diffeomorphism $\Phi$ from $M$ onto $\hat{M}$ and $g, \hat{g}$ share $q$ distinct points $a_1, a_2, \cdots, a_q$,  
    then the following statements hold:
    \begin{itemize}
     \item If $q\ge 7$ and either $M$ or $\hat{M}$ is complete, then $g\equiv \hat{g}\circ\Phi$.  
     \item If $q\ge 6$ and both  $M$ and $\hat{M}$ are complete and have finite total curvature, then $g\equiv \hat{g}\circ\Phi$. 
    \end{itemize}
\end{thm}
The number $7$ is the best possible since H. Fujimoto construct two mutually distinct  isometric complete minimal surfaces whose Gauss maps are distinct and have the same inverse images for  six points. Later, H. Fujimoto \cite{MR1243805}, J. Park and M. Ru \cite{MR3667235} gave generalizations of the unicity theorem to minimal surfaces in $\mathbb R^n (n>3)$.  After that, many mathematicians studied unicity theorem for Gauss maps of minimal surfaces (for more details, see \cite{MR3784760, MR3690412, MR2373945, MR4176747}).

 It is quite natural to study the value distribution problem for Gauss maps of complete space-like stationary surfaces (i.e. surfaces with zero mean curvature and positive-definite metric) in $n$-dimensional Lorentz space $\mathbb R^{n-1,1}$. E. Calabi \cite{Calabi} showed that any complete maximal space-like surface in $\mathbb R^{2,1}$ has to be affine linear, that is the Gauss map must be constant. Furthermore, the ramification problem and the unicity problem of weakly complete maxfaces in $\mathbb R^{2,1}$ are investigated by Y. Kawakami in \cite{MR3415658}. Here, maxfaces are maximal (or stationary) space-like surfaces with some admissible singularities, as
introduced by Umehara and Yamada \cite{MR2225080}.
 In \cite{Ou}, the author, C, Cheng and L. Yang \cite{Ou} proved the exceptional value theorem for Gauss maps of complete space-like stationary surfaces in $\mathbb R^{3,1}$:
 \begin{thm}
 Let M be a non-flat complete space-like stationary surface in $\mathbb R^{3,1}$, $(\psi_1,\psi_2)$ be the Gauss map of $M$, and $q_i$ be the number of exceptional values of $\psi_i(i=1,2)$.  If neither $\psi_1$ nor $\psi_2$ are constant, then $\min\{q_1,q_2\}\le 3$ or $q_1=q_2=4$.
 \end{thm}
 This result cannot be further improved without additional assumptions, since every minimal surface in $\mathbb R^3$ is a space-like stationary surface in $\mathbb R^{3,1}$. They also considered the situation that Gauss image lies in a graph of a rational function of degree $m$, which contain the degenerate cases $(m\le 1)$. They  obtained the following result: 
 \begin{thm}\cite{Ou}
 \label{ethm}
 Let M be a non-flat complete space-like stationary surface in $\mathbb R^{3,1}$, whose Gauss map $(\psi_1,\psi_2)$ satisfies $\psi_2=f(\psi_1)$ with $f$ a rational function of degree $m$, then the number $q_1$ of the exceptional values of $\psi_1$ should satisfy $|E_f|\le q_1\le m-|E_f|+3$, where $E_f=\{z\in \bar{\mathbb C}, f(z)=\bar z\}$.
     
 \end{thm}

  In this paper, we study the ramification and the unicity problem for Gauss maps of complete space-like stationary surfaces in $\mathbb R^{3,1}$. Since every complete space-like stationary surface in $\mathbb R^{3,1}$ corresponds to a complete regular minimal surface in $\mathbb R^4$ which has same Gauss map, thus we can easily get Theorem \ref{rami}  and  Theorem \ref{unicity} for  ramification problem and unicity problem respectively. 

  Moreover, we  extend the ramification result of H. Fujimoto \cite{MR1167375}  and the unicity theorem of H. Fujimoto \cite{MR1219880} to the complete space-like stationary surfaces in $\mathbb R^{3,1}$, whose Gauss map satisfies $\psi_2=f(\psi_1)$ where $f$ is a rational function with degree $m$. In the case $m\le 1$, $M$ is called degenerate, which contains the minimal surfaces in $\mathbb R^3$, maximal surfaces in $\mathbb R^{2,1}$, 2-degenerate space-like stationary surface, and three types of space-like stationary graphs in $\mathbb R^{3,1}$ (see more details in \cite{Ou}).
  
We first give some basic preliminaries in section \ref{sec2}.  And in section \ref{sec3}, we study the ramification problem of Gauss map for space-like stationary surface in $\mathbb R^{3,1}$ and get the following theorem:
  
  \begin{mthma}
  Let $M\subset \mathbb R^{3,1}$ is a non-flat complete space-like stationary surface with Gauss map $G=(\psi_1,\psi_2)$ satisfying $\psi_2=f(\psi_1)$, where $f$ is a rational function of degree $m$. Set $E_f=\{z\in \bar{\mathbb C}, f(z)=\bar z\}$ and assume that there are $q(q>m-2|E_f|+3)$ distinct points $a_1,\cdots, a_q \in \bar{\mathbb C}\setminus E_f$, such that the Gauss map $\psi=\psi_1$ of $M$ is ramified over $a_j$ with multiplicity at least $e_j$ for each $j$, then 
\begin{equation} \label{mr}  
\gamma:=|E_f|+\sum_{j=1}^q(1-\frac{1}{e_j})\le m-|E_f|+3.
\end{equation}     
\end{mthma}
In equation $(\ref{mr})$, if all $e_j(1\le j\le q)$ are $\infty$, then we can easily get the Theorem \ref{ethm}.

In section \ref{sec4}, we study the unicity problem for Gauss maps of complete stationary space-like surfaces in $\mathbb R^{3,1}$ and get the following theorem:

 \begin{mthmb}
Let $M$ and $\hat{M}$ be two complete non-flat space-like stationary surfaces in $\mathbb R^{3,1}$ with Gauss maps $G=(\psi_1,\psi_2)$ and $\hat{G}=(\hat{\psi}_1, \hat{\psi}_2)$ respectively. Suppose $\psi_2=f(\psi_1)$ and $\hat{\psi}_2=f(\hat{\psi}_1)$ with $f$ a rational function of degree $m$ and $\Phi: M\to \hat{M}$ is  a conformal diffeomorphism. Assume there are $q$ points $a_1,a_2,\cdots, a_q\in \bar{\mathbb C}$ such that 
$$\psi_1^{-1}(a_i)=(\hat{\psi}_1\circ \Phi)^{-1}(a_i),  i=1,\cdots, q.$$ 
Then, we have $\psi_1\equiv \hat{\psi}_1\circ \Phi$ if  $q\ge m-|E_f| +6 $, 
where $E_f=\{z\in \bar{\mathbb C}, f(z)=\bar{z}\}$.
    
\end{mthmb}

\section{Preliminaries}\label{sec2}

\subsection{The Gauss map of space-like stationary surfaces in $\mathbb R^{3,1}$}\label{subsec2}
Denote $\mathbb R^{3,1}$ be the $4$-dimensional  Minkowski space with the  Minkowski inner product:
	\begin{equation*}
	\langle \mathbf{u},\mathbf{v}\rangle=u_1v_1+u_2v_2+u_3v_3-u_4v_4,
	\end{equation*}
 where $\mathbf{u}=(u_1,u_2,u_3,u_4), \mathbf{v}=(v_1,v_2,v_3,v_4) \in \mathbb R^{3,1}$. 
	$\mathbf{u}\in \mathbb R^{3,1}$ is  space-like if $\langle \mathbf{u},\mathbf{u}\rangle>0$; $\mathbf{u}$ is  time-like if $\langle \mathbf{u},\mathbf{u}\rangle<0$; $\mathbf{u}$ is called a null vector or a light-like vector if $\langle \mathbf{u},\mathbf{u}\rangle=0$.

 Let $\mathbf{x}:M\to \mathbb R^{3,1}$ be an oriented space-like stationary surface in the Minkowski space. That is the mean curvature vector field $\mathbf H$  of $M$ vanishes everywhere, and  the pull-back metric $ds^2=\langle d\mathbf{x},d\mathbf{x}\rangle$	is positive-definite everywhere. 
$M$ is stationary if and only if the restriction of each coordinate function on $M$ is harmonic.

Let $(u,v)$ be  local isothermal parameters in a neighborhood of $p\in M$, then the tangent space at $p$  is $T_pM=\text{span}\{\mathbf{x}_u,\mathbf{x}_v\}$. 
The Gauss map of $M$ is defined by
$$G: p\in M\mapsto T_p M\in \mathbf{G}_{1,1}^2,$$
where $ \mathbf{G}_{1,1}^2$ is the Lorentz--Grassmann manifold consisting of all oriented space-like 2-plane in $\mathbb R^{3,1}$. 
 Denote
	\begin{equation*}
	    Q_{1,1}=\{[\mathbf z]\in \mathbb{CP}^3:z_1^2+z_2^2+z_3^2-z_4^2=0\},
	\end{equation*}
	\begin{equation*}
	Q_{1,1}^+=\{[\mathbf z]\in Q_{1,1}:|z_1|^2+|z_2|^2+|z_3|^2-|z_4|^2>0\}.	\end{equation*}
Since $\langle \mathbf{x}_u,\mathbf{x}_u\rangle=\langle\mathbf{x}_v,\mathbf{x}_v\rangle>0$, and  $\langle \mathbf{x}_u,\mathbf{x}_v\rangle=0$, then $[\mathbf{x}_z]=[\frac{1}{2}(\mathbf{x}_u-i\mathbf{x}_v)]\in Q_{1,1}^+$, where $z=u+iv$. We may identify $\mathbf{G}_{1,1}^2$ with $Q_{1,1}^+$ via the map
$$i: \Pi=\text{span}\{\mathbf{u},\mathbf{v}\}\in \mathbf{G}_{1,1}^2\mapsto [\mathbf{z}]=[\mathbf{u}-i\mathbf{v}]\in Q_{1,1}^+.$$

\begin{pro}
		\cite{Ou}
	\label{com_str}
		For $Q_{1,1}$ and $Q_{1,1}^+$, we have:
		\begin{enumerate}
			\item[(1)] $Q_{1,1}$ is biholomorphic to $\bar{\mathbb C}\times \bar{\mathbb C}=\mathbb{S}^2\times \mathbb{S}^2$, with $\bar{\mathbb C}=\mathbb C\cup \{\infty\}$
			the extended complex plane.
			\item[(2)] $Q_{1,1}^+$ is biholomorphic to $\{(w_1,w_2)\in \bar{\mathbb C}\times \bar{\mathbb C}:w_2\neq \bar{w}_1\}$, where $\overline{\infty}=\infty$.
		\end{enumerate}
	\end{pro}

 The holomorphic map $\Psi: Q_{1,1}\to \bar{\mathbb C}\times \bar{\mathbb C}$ is defined by
 \begin{itemize}
     
     \item $\Psi([\mathbf z])=(\frac{z_1+iz_2}{z_3+z_4},  \frac{z_1-iz_2}{z_3+z_4})$ whenever $z_3+z_4\ne 0$;

     \item $\Psi([\mathbf z])=(\frac{z_4-z_3}{z_1-iz_2}, \infty)$ whenever $z_3+z_4=z_1+iz_2 \equiv 0\text{ and } z_1-iz_2\ne 0$;

      \item $\Psi([\mathbf z])=(\infty, \frac{z_4-z_3}{z_1+iz_2})$ whenever $z_3+z_4=z_1-iz_2 \equiv 0\text{ and } z_1+iz_2\ne 0$;
      
      \item $\Psi([\mathbf z])=(\infty, \infty)$ whenever $z_3+z_4=z_1-iz_2=z_1+iz_2 \equiv 0$.
 \end{itemize}
 The metric of $Q_{1,1}^+$ in terms of $w_1$ and $w_2$ is :
	
	\begin{equation}\label{g2}
	g=\text{Re}\left[\frac{4d\bar{w}_1dw_2}{(\bar{w}_1-w_2)^2}\right].
	\end{equation}

\subsection{Weierstrass representation of space-like stationary Surfaces in $\mathbb R^{3,1}$}\label{subsubsec2}
Denote
	\vspace{-6pt}
	\begin{equation*}
	\varphi=(\varphi_1,\varphi_2,\varphi_3,\varphi_4):=\Big(\frac{\partial x_1}{\partial z},\frac{\partial x_2}{\partial z},\frac{\partial x_3}{\partial z},\frac{\partial x_4}{\partial z}\Big)dz.
	\end{equation*}
	Then, the harmonicity of $x_k$ forces $\varphi_k$ to be a holomorphic 1-form that can be globally defined on $M$.
	$[\mathbf{x}_z]\in Q_{1,1}^+$ is equivalent to saying that
	\vspace{-6pt}
	\begin{eqnarray}
	\varphi_1^2+\varphi_2^2+\varphi_3^2-\varphi_4^2&=&0,\label{phi1}\\
	|\varphi_1|^2+|\varphi_2|^2+|\varphi_3|^2-|\varphi_4|^2&>&0\label{phi2}.
	\end{eqnarray}	
	Let $$(\psi_1,\psi_2)=\Psi(\varphi),\quad  dh=\frac{1}{2}(\varphi_3+\varphi_4).$$ Then, $\psi_1,\psi_2$ are meromorphic functions on $M$, and $dh$ is a holomorphic 1-form on $M$. 
	If $dh \equiv 0$, then (\ref{phi1}) implies 
 $$\varphi_1-i\varphi_2\equiv 0 \text{ or } \varphi_1+i\varphi_2\equiv 0.$$
	Hence, $\varphi=(\varphi_1,\pm i\varphi_1,\varphi_3,-\varphi_3)$ and the zeros of $\varphi_1$ and $\varphi_3$ do not coincide.
	Otherwise, the zeros of $dh$ are discrete. 

Thus, the Weierstrass representation of space-like stationary surfaces in $\mathbb{R}^{3,1}$ can be stated as follows:
	\begin{thm}\cite{ma-wang1}
	Let $\psi_1,\psi_2$ be meromorphic functions and $dh$ be a holomorphic 1-form on a Riemann surface $M$. If $\psi_1, \psi_2, dh $ satisfy the regularity conditions $(1), (2)$ and the period condition $(3)$:
	\begin{enumerate}

	    \item [(1)] $\psi_1\neq\bar{\psi_2}$ on $M$, and their poles do not coincide;

	    \item [(2)]  The zeros of $dh$ coincide with the poles of $\psi_1$ or $\psi_2$ with the same order; 
	
	    \item [(3)]  For any closed curve $C$ on $M$:
	    \vspace{-6pt}
	\begin{equation*}
	\int_{C}\psi_1dh=-\overline{\int_{C}\psi_2 dh}\quad,\quad Re\int_{C}dh=0=Re\int_{C}\psi_1\psi_2 dh\ .\label{c7}
		\end{equation*}
		\end{enumerate}
		Then, the 
 following equation defines a space-like stationary surface 
$\mathbf x:M\rightarrow\mathbb{R}^{3,1}$:
	\vspace{-6pt}
\begin{equation}
\mathbf x=2Re\int\big(\psi_1+\psi_2,-i(\psi_1-\psi_2),1-\psi_1\psi_2,1+\psi_1\psi_2\big)dh\ .\label{c8}
\end{equation}
			
Conversely, every space-like stationary surface $\mathbf x:M\rightarrow\mathbb{R}^{3,1}$ can be represented as equation (\ref{c8}), where $dh$, $\psi_1$, $\psi_2$ satisfy the conditions $(1), (2)$, and  $(3)$.
	\end{thm}
	The induced metric is: 
	\vspace{-6pt}
		\begin{equation}\label{phi5}
	ds^2=\langle \mathbf{x}_z,\mathbf{x}_{\bar{z}}\rangle=\langle \varphi,\bar{\varphi}\rangle=2|\psi_1-\bar{\psi_2}|^2|dh|^2.
	\end{equation}
	
Actually, as shown in \cite{Ou},  $\psi_1$ and $\psi_2$ correspond to  two null vectors $\mathbf{y},\mathbf{y}^*$ respectively, which  span the normal plane of $M$ at the considered point. Moreover, $\psi_2\neq \bar{\psi}_1$ is equivalent to saying that $\mathbf{y}\neq \mathbf{y}^*$
	everywhere. 
	
\begin{rem}
If we let $\psi_1\equiv-\frac{1}{\psi_2}$, (\ref{c8}) yields the representation for a minimal surface in $\mathbb R^3$. If $\psi_1\equiv\frac{1}{\psi_2}$, then  we can obtain the Weierstrass representation for a maximal surface in $\mathbb R^{2,1}$. Actually, all minimal surfaces in $\mathbb R^3$ and maximal surfaces in $\mathbb R^{2,1}$ are space-like stationary surfaces in $\mathbb R^{3,1}$.
\end{rem}


\begin{rem}\label{mob}
The induced action on $\mathbb S^2$  of the Lorentz transformation is just the M\"{o}bius transformation on  $\mathbb S^2$ (i.e., a fractional linear transformation on $ \bar{\mathbb C}$). Since $(\psi_1, \psi_2)=\Psi([\mathbf{x}_{z}])$, then  the Gauss maps $\psi_1,\psi_2$ and the height differential $dh$  can be transformed as below:
\vspace{-6pt}
$$\psi_1 \Rightarrow \frac{a\psi_1+b}{c\psi_1+d}, \quad \psi_2\Rightarrow \frac{\bar{a}\psi_2+\bar{b}}{\bar{c}\psi_2+\bar{d}},\quad dh\Rightarrow (c\psi_1+d)(\bar{c}\psi_2+\bar{d})dh,$$
where $S=\left(\begin{array}{cc}
 a & b\\
 c & d
\end{array}\right)\in SL(2,\mathbb{C})$.
\end{rem}

Let
$$\varphi^*_k=\varphi_k, k=1,2,3,\quad \varphi^*_4=i\varphi_4,$$
we have 
\begin{equation*}
   \sum_{k=1}^{4}(\varphi^*_k)^2=\sum_{k=1}^{4}\varphi^2_k-\varphi_4^2=0,
\end{equation*}
and 
\begin{equation} \label{var}
\sum_{k=1}^{4}|\varphi^*_k|^2=\sum_{k=1}^{4}|\varphi_k|^2 \ge \sum_{k=1}^{3}|\varphi_k|^2-|\varphi_4|^2>0.
\end{equation}
Then 
\begin{equation*}
    \mathbf x^*=Re\int(\varphi^*_1,\varphi^*_2,\varphi^*_3,\varphi^*_4 )dz
\end{equation*}
defines a simply-connected regular minimal surface $M^*$ in $\mathbb R^4$. Conversely, given a simply-connected regular minimal surface $M^*$ in $\mathbb R^4$, the corresponding stationary surface $M$ may not be space-like. Thus $\mathbf x\leftrightarrow \mathbf x^*$  gives a one-to-one correspondence between all simply-connected generalized stationary surface $M$ in $\mathbb R^{3,1}$ and all simply-connected generalized minimal surface $M^*$ in $\mathbb R^4$.

If $M$ is complete, $M^*$ is also complete by the equation (\ref{var}). Thus, every simply-connected complete space-like stationary surface in $\mathbb R^{3,1}$  can be viewed as a simply-connected complete regular minimal surface in $\mathbb R^4$.

\subsection{Space-like stationary surfaces with rational graphical Gauss image}\label{subsec4}

Let $M$ be a space-like stationary surface in $\mathbb R^{3,1}$. If the Gauss image $G(M)$ of $M$ lies in a hyperplane $H_A\subset Q_{1,1}$	with $[A]\in \mathbb {CP}^{2,1}$, then $M$ is called {\bf degenerate}. If $M$ is 1-degenerate, then the Gauss map $(\psi_1,\psi_2)$ satisfies $\psi_2=M_S(\psi_1):=\frac{a\psi_1+b}{c\psi_1+d}$, where $S:=\left(\begin{array}{cc}
     a & b  \\
    c & d
\end{array}\right)\in SL(2,\mathbb C)$. $\psi_1\ne \bar{\psi_2}$ implies that $\psi_1$ must omit all the points in $E_S$, where 
$$E_S=\{z\in \overline{\mathbb C}, M_S(z)=\bar{z}\}.$$ 
Under the conjugate similarity equivalence relation ($S_1\stackrel{conj}{\sim}S_2 \text{ if and only if 
 } S_2=\pm \bar{T}S_1T^{-1}$, where $T=\left(\begin{array}{cc}
     a & b  \\
    c & d
\end{array}\right)\in SL(2,\mathbb C)$. See section 2.4 in \cite{Ou} for more details ), $M$ is one of the following  types:
\begin{enumerate}
\item $M$ is a maximal surface in $\mathbb R^{2,1}$. In this case,  $S\stackrel{conj}{\sim} \left(\begin{array}{cc}
    1 & 0  \\
    0 & 1
\end{array}\right) $ and $|E_S|=\infty$ ; 
    \item $M$ is a  degenerate space-like stationary surface of hyperbolic type, which is congruent to entire space-like stationary graph of $F:\mathbb R^2\to \mathbb R^{1,1}$. In this case,  $S\stackrel{conj}{\sim} \left(\begin{array}{cc}
     e^u & 0  \\
    0 & e^{-u}
\end{array}\right) $ with $u\in(0, \infty)$  and $|E_S|=2$ ;
   
     \item $M$ is a minimal surface in $\mathbb R^3$ or a  degenerate space-like stationary surface of elliptic type, which can be deformed by a minimal surface in $\mathbb R^3$. In this case, $S\stackrel{conj}{\sim} \left(\begin{array}{cc}
     0 & ie^{-i\alpha}  \\
    ie^{i\alpha} & 0
\end{array}\right) $ with $\alpha \in (0,\frac{\pi}{2}]$  and  $|E_S|=0$;
     
      \item $M$ is a  degenerate space-like stationary surface of parabolic  type. In this case, $S\stackrel{conj}{\sim} \left(\begin{array}{cc}
     1 & 1\\
    0 & 1
\end{array}\right) $ and  $|E_S|=1$.
\end{enumerate}

If $M$ is 2-degenerate, that is the Gauss image $G(M)$ lies in the intersection of two linear independent hyperplanes in $Q_{1,1}$. Then $\psi_1$ or $\psi_2$ is a constant function. By  Theorem 4.3 of \cite{Ou},  $M$ has to be an entire graph $F:(x_1,x_2)\in \mathbb R^2\to h(x_1,x_2)\mathbf y_0\in \mathbb R^{1,1}$, where $h(x_1,x_2)$ is a harmonic map and $\mathbf y_0$ is a null vector. Since the Gauss curvature $K\equiv 0$, the universal covering space $\Tilde{M}$ of $M$ is conformally equivalent to the whole complex plane $\mathbb C$.

More generally, if the Gauss map $(\psi_1,\psi_2)$ satisfies $\psi_2=f(\psi_1)$ with $f$  a rational function of degree $m$, then $\psi_1\ne \bar{\psi_2}$ implies that $\psi_1$ must omit all the points in $E_f$, where  
$$E_f=\{z\in \overline{\mathbb C}, f(z)=\bar{z}\}. $$

Then we have the following results:
\begin{pro}\cite{Ou}
Let $M$ be a non-flat complete space-like stationary surface in $\mathbb R^{3,1}$. If the Gauss map $(\psi_1,\psi_2)$ satisfies $\psi_2=f(\psi_1)$ with $f$  a rational function of degree $m$, then
\begin{itemize}
\item $m\le 5;$\\
    
    \item If $m\ge 2$, then $m-1\le |E_f|\le \frac{m+3}{2}$. \\
 \item   The number of exceptional values of $\psi_1$ in $\bar{\mathbb C}$ cannot exceed $m-|E_f|+3$.  
\end{itemize}
\end{pro}


\subsection{Metrics with negative curvature}

 Let $\Omega$  be a domain in $\mathbb{C}$ of hyperbolic type which can be endowed with a Poincar$\acute{e}$ metric  denoted by
	$$ds^2=\lambda_{\Omega}(z)^2|dz|^2, $$
	where $\lambda_{\Omega}(z)$ is a positive $C^2$-function on $\Omega$ satisfying the condition 
 $\Delta log\lambda_{\Omega}=\lambda_{\Omega}^2.$ In particular, for a disc $\mathbb{D}(R)=\{z\in \mathbb C,|z|<R\}$, we have 
	$$\lambda_{\mathbb{D}(R)}(z)=\frac{2R}{R^2-|z|^2}.$$
	\begin{thm}
		\cite{L.A.Ahlfors}\label{12}
		Let $\Omega $ be a domain in $\mathbb{C}$ and $\lambda$ be a positive $C^2$-function on $\Omega$ satisfying the condition $\Delta log\lambda\ge \lambda^2$. Then, for every holomorphic map $f:\mathbb{D}(R)\to \Omega,$
		\begin{equation}
		|f'(z)|\lambda(f(z))\le \frac{2R}{R^2-|z|^2}.
		\label{Schwz lem}
		\end{equation}
Especially, if $f(z)=z$, then we get
$$\lambda(z)\le\frac{2R}{R^2-|z|^2}.$$
	\end{thm}

\section{Ramification problem }\label{sec3}

 Let $M$  be a  space-like stationary surfaces in $\mathbb R^{3,1}$ with Gauss map $G=(\psi_1,\psi_2)$, and the induced metric  $(\ref{phi5})$ satisfies 
 $$ds^2=2|\psi_1-\bar{\psi_2}|^2|dh|^2\le 2(1+|\psi_1|^2)(1+|\psi_2|^2)|dh|^2=d\Tilde{s}^2.$$
Then $(M, d\Tilde{s}^2)$ is a complete Riemann surface if $ds^2$ is complete. Thus the following ramification result for Gauss map of stationary surfaces in $\mathbb R^{3,1}$ can be derived from Theorem \uppercase\expandafter{\romannumeral3}  in \cite{MR0982173}   :

\begin{thm}\label{rami}

Let $M\subset \mathbb R^{3,1}$ is a non-flat complete space-like stationary surface with Gauss map $G=(\psi_1,\psi_2)$. Then
if $\psi_1,\psi_2$ are not constant, and $\psi_i(i=1,2)$ is ramified over $a^{ij}(1\le j\le q_i)$ with multiplicity at least $m_{ij}$, then
$\min\{\gamma_1, \gamma_2\}\le 3$, or $\gamma_1=\gamma_2=4$, where
    $$\gamma_1:=\Sigma_{j=1}^{q_1}(1-\frac{1}{m_{1j}}), \quad \gamma_2:=\Sigma_{j=1}^{q_2} (1-\frac{1}{m_{2j}}).$$
\end{thm}

\begin{rem}
  If all $m_{ij}=\infty(i=1,2, j=1, \cdots, q_i)$, that is $a^{ij}$ are exceptional values of $\psi_i$. Then $\gamma_i=q_i (i=1,2)$ are the number of exceptional values, and the above theorem is a generalization of Theorem $1.2$ in \cite{Ou}.
\end{rem}

Next, we will discuss the ramification problem 
with  condition $\psi_2=f(\psi_1)$ where $f(z)=\frac{P(z)}{Q(z)}$  is a rational function of degree $m$. Throughout the proof process, we may assume $M$ is simply connected, otherwise we consider its universal covering space. By Koebe's uniformization  theorem, $M$ is conformally  equivalent to  the whole complex plane $\mathbb C$ or to the unit disc $\mathbb D$. For the case $M=\mathbb C$, $\gamma$ defined in (\ref{mr}) satisfies $\gamma\le 2$ by the following theorem: 

\begin{thm}\cite{MR1546112}\label{gpt}
Let $f$ be a non-constant meromorphic function over $\mathbb C$, if there are $q$ distinct points $a_1,a_2,\cdots, a_q \in \overline {\mathbb C}$ such that all the roots of the equation $f(z)=a_i$ have multiplicity at least $e_i$($e_i=\infty$ if $f(z)=a_i$ has no root). Then  
\begin{equation}
    \sum_{i=1}^q(1-\frac{1}{e_i})\le 2.
\end{equation}
\end{thm}

\begin{pro}\label{special}
\begin{itemize}
 
 \item If  $M$ is 2-degenerate $(m=0)$ or  1-degenerate of hyperbolic type($m=1, |E_f|=2$), $M$ is an entire graph over $\mathbb R^2$ and the universal covering space $\Tilde{M}$ of $M$ is conformally equivalent to the whole complex plane. Then  $\gamma\le 2$ by Theorem \ref{gpt}.

\item If $M$ is 1-degenerate of elliptic type($m=1, |E_f|=0$), $M$ can be deformed by a minimal surface in $\mathbb R^3$, thus  $\gamma\le 4$ by Theorem \ref{r3}.
\end{itemize}
\end{pro}
Thus, we just need to consider the 1-degenerate of parabolic type ($m=1, 
|E_f|=1$) and the case $2 \le m \le 5$. 
 Since $E_f\ne\emptyset $, without loss of generality, we can assume $\infty\in E_f$ under a suitable M\"{o}bius transformation. This means $f(\infty)=\infty$, hence 
$$m=degree P>degree Q =n.$$
Set 
$$E_f=\{c_1,c_2,\cdots, c_l, \infty\}$$
and let $b_1,\cdots, b_s$ be the zeros of $Q(z)$ with orders $r_1,\cdots, r_s$ respectively. Then
$$|f(z)-\bar{z}|\le C\left(\Pi_{i=1}^l|z-c_i|\right)\left(\Pi_{j=1}^s|z-b_j|^{-r_j}\right)(1+|z|^2)^{\frac{m-l}{2}}.$$

Let $\psi:=\psi_1$, the induced metric on $M$ is

\begin{equation}
\begin{split}\label{metric}
    ds^2 &=|f(\psi)-\bar{\psi}|^2|dh|^2\\
         & \le \left(\Pi_{i=1}^l|\psi-c_i|^2\right)\left(\Pi_{j=1}^s|\psi-b_j|^{-2r_j}\right)(1+|\psi|^2)^{m-l}|dh|^2\\
         &=\left(\Pi_{i=1}^l|\psi-c_i|^2\right)(1+|\psi|^2)^{m-l}|\omega|^2=d\Tilde{s}^2,
  \end{split}
\end{equation}
 then $d\Tilde{s}^2$ is also a complete metric on $M$, where $$\omega=\frac{dh}{\Pi_{j=1}^s|\psi-b_j|^{r_j}}$$ is a holomorphic 1-form on $M$ with no zero.
 Since 
 $$m-1\le |E_f|=l+1\le \frac{m+3}{2},$$
 we have $0<m-l\le 2$.

For a non-zero meromorphic function $f$, the divisor $\nu_f$ of $f$ is defined as follows:
\begin{equation*}
    \nu_f(a)=\begin{cases}
    k & \text{ if $f$ has a zero of order $k$ at $a$, }  \\
    -p & \text{ if $f$ has a pole of order $p$ at $a$, }\\
    0 & \text{ otherwise.}\end{cases}
\end{equation*}



\begin{lem}
 Assume that there are $q$ distinct points $a_1,\cdots, a_q \in \bar{\mathbb C}\setminus E_f$, such that the meromorphic function $\psi$ on $\mathbb D(R)$ is ramified over $a_j$ with multiplicity at least $e_j$ for each j.  If $l+q-\sum_{j=1}^q\frac{1}{e_j}-1>\epsilon>(l+q)\epsilon'>0$, where $0<\epsilon'<1$.  Set 
 \begin{equation}
     v=\frac{(1+\psi^2)^{\frac{(m-l)p}{2}}\psi'}{\Pi_{i=1}^l(\psi-c_i)^{1-\epsilon'}\Pi_{j=1}^q(\psi-a_j)^{1-\frac 1 e_j-\epsilon'}}
 \end{equation}
on $\mathbb D(R)\setminus \{z, \psi(z)=a_j \text{ for some j}\}$  and $v=0$ on $\mathbb D(R)\cap \{z, \psi(z)=a_j \text{ for some j}\}$. Set $$p=\frac{l+q-\sum_{j=1}^q\frac 1 e_j-1-\epsilon}{m-l}.$$

Then $v$ is a continuous function in $\mathbb D(R)$, and there exists a positive constant $B$ such that \begin{equation}\label{neq}
    |v|\le B\frac{2R}{R^2-|z|^2}.
\end{equation}
\end{lem}


   

\begin{proof}
First, we prove the continuity of $v$. Obviously, $v$ is continuous on 
$$\mathbb D(R)\setminus \{z, \psi(z)=a_j \text{ for some } 0\le j\le q\}.$$
Take a point $z_0$ with $\psi(z_0)=a_j$ for some $j$. 
Now, we write
$$\psi(z)-a_j=(z-z_0)^{r_0}h(z), $$
where $r_0=\nu_\psi(z_0)\ge e_j$ and $h(z)$ is a meromorphic function with $h(z_0)\ne 0$. Then 
\begin{equation}\label{order}
    \psi'=(z-z_0)^{r_0-1}(r_0h(z)+(z-z_0)h'(z)),
\end{equation}
that is $\nu_{\psi'}(z_0)=\nu_\psi(z_0)-1=r_0-1$.
Since 
\begin{align*}   
\nu_{v}(z_0)&=\nu_{\psi'}(z_0)-(1-\frac 1 e_j-\epsilon')\nu_\psi(z_0)\\
&=r_0-1-r_0(1-\frac 1 e_i-\epsilon')\\
&=\frac{r_0}{e_i}-1+r_0\epsilon'>0,
\end{align*}
then $\lim_{z\to z_0}v=0$. This implies that $v$ is continuous on $\mathbb{D}(R)$. 


Obviously, (\ref{neq}) holds if $z\in \mathbb D(R)\cap \{z, \psi(z)=a_j \text{ for some j}\}$. 
 If $z\in \mathbb D(R)\setminus \{z, \psi(z)=a_j \text{ for some j}\}$,
 let $\Omega=\mathbb{C}\setminus \{a_1,\dots , a_q, c_1,\dots, c_l\}$ and $\lambda_{\Omega}$ be the Poincar$\acute{e}$ metric of $\Omega$, that is $\lambda_{\Omega}$ satisfies $\Delta log\lambda_{\Omega}=\lambda_{\Omega}^2,$  then (\ref{Schwz lem})  implies that 
		\begin{equation}
		|\psi'(z)|\lambda_{\Omega}(\psi(z))\le \frac{2R}{R^2-|z|^2}.
		\label{psilambda_}
		\end{equation}
  Moreover, $\lambda_{\Omega}$ satisfies(see e.g. p.250 of \cite{Nevanlinna})
  $$\lambda_{\Omega}(\psi)\sim \frac{1}{|\psi-a_i|\log|\psi-a_i|^{-1}}\quad \text{ near } a_i,$$
  $$\lambda_{\Omega}(\psi)\sim \frac{1}{|\psi-c_j|\log|\psi-c_j|^{-1}}\quad \text{ near } c_j,$$
  $$\lambda_{\Omega}(\psi)\sim \frac{1}{|\psi|\log|\psi|}\quad \text{ near }  \infty .$$
  Denote 
  $$u=\frac{(1+|\psi|^2)^{\frac{(m-l)p}{2}}}{\Pi_{i=1}^l|\psi-c_i|^{1-\epsilon'}\Pi_{j=1}^q|\psi-a_j|^{1-\frac 1 e_j-\epsilon'}\lambda_{\Omega}(\psi)},$$
then by a direct calculation, we get $u$ is bounded by a constant $B$ .

Thus by Theorem \ref{12}, we get $$v=u|\psi'|\lambda_{\Omega}(\psi)\le B\frac{2R}{R^2-|z|^2}.$$
\end{proof}

Now, we begin to  proof the Main Theorem A.
\begin{proof}
 If the equation (\ref{mr}) is wrong, then 
 $$l+1+\sum_{j=1}^q(1-\frac{1}{e_j})>m-l+2,$$
 that is 
\begin{equation}\label{wr}
  2l+q-\sum_{j=1}^q\frac{1}{e_j} -m-1>0.  
 \end{equation}
 Take $\epsilon'$ with 
 \begin{equation*}
   \frac{2l+q-\sum_{j=1}^q\frac{1}{e_j} -m-1}{l+q}>\epsilon'>\frac{2l+q-\sum_{j=1}^q\frac{1}{e_j} -m-1}{q+m},
 \end{equation*}
and set 
$$\Tilde{p}=\frac{1}{p}=\frac{m-l}{l+q-\sum_{j=1}^q\frac{1}{e_j}-1-\epsilon}.$$
Then 
\begin{equation*}
0<\Tilde{p}=\frac{m-l}{l+q-\sum_{j=1}^q\frac 1 {e_j}-1-\epsilon}<\frac{m-l}{m-l+2-\epsilon}<1,
\end{equation*}
 and
\begin{align*}  
\frac{\Tilde{p}}{1-\Tilde{p}}>\frac{\epsilon'\Tilde{p}}{1-\Tilde{p}}&=\frac{(m-l)\epsilon'}{l+q-\sum_{j=1}^q\frac{1}{e_j}-1-\epsilon-m+l}\\
&=\frac{(m-l)\epsilon'}{2l+q-\sum_{j=1}^q\frac{1}{e_j}-1-m-\epsilon}\\
&>\frac{(m-l)\epsilon'}{(m+q)\epsilon'-\epsilon}\\
&=\frac{(m-l)\epsilon'}{(m-l)\epsilon'+(l+q)\epsilon'-\epsilon}>1.
\end{align*}


Now, we consider $M=\mathbb D$ and $$E=\{z\in \mathbb D, \psi'(z)=0\}.$$
Non-flatness of $M$ implies $E\ne \mathbb D$. 

Let $\omega=gdz$, where $g$ is a holomorphic function on $\mathbb D$ with no zero, and consider the following many-valued function:

$$\eta=\frac{g^{\frac{1}{1-\Tilde{p}}}\Pi_{i=1}^l(\psi-c_i)^{\frac{\Tilde{p}(1+p-\epsilon')}{1-\Tilde{p}}}\Pi_{j=1}^q\left(\psi-a_j\right)^{\frac{\Tilde{p}(1-\frac{1}{e_j}-\epsilon')}{1-\Tilde{p}}}}{\psi'^{\frac{\Tilde{p}}{1-\Tilde{p}}}}$$
on $\mathbb D'=\mathbb D\setminus E$.


Take an arbitrary single-value branch of $\eta$, still denoted by $\eta$ for the matter of convenience.
Let $$w=F(z)=\int\eta dz$$
be a holomorphic mapping from $\mathbb D'$ to $\mathbb C$, satisfying $F(0)=0, F'(z)=\eta(z)\ne 0$. Hence there exists a holomorphic inverse mapping $z=H(w)$ on a neighborhood of $0$. Let $\mathbb{D}(R)=\{w:|w|\le R\}$ is the largest ball that $H$ can be defined, then $R<+\infty$ (otherwise, $H$ is a non-constant bounded entire function, which contracts to the Liouville theorem) and
		there exists a point $a$ on the boundary of $\mathbb{D}(R)$, such that $H$ cannot be extended beyond a neighborhood of $a$. Let
		$l_a:=\{ta:0\leq t< 1\}$ be the straight line segment starting from $0$, then $H(l_a)$
		must be a divergent curve in $\mathbb{D}'$ as $t$ tends to $1$. 
  \begin{lem}
There exists a point with $|a|=R$ such that $H(l_a)$ is a divergent curve in $\mathbb D$.   
\end{lem}
\begin{proof}
If $H(l_a)$ tends to a point $z_0$ such that $\psi'(z_0)=a_j$ for some $0\le j\le q$, then we have two cases:
  
  \textbf{Case 1: }$\psi(z_0)\ne a_j$ for all $ 1\le j\le q$,  then $\nu_{\eta}(z_0)=-\frac{k\Tilde{p}}{1-\Tilde{p}}$.
 Since 
 $$\frac{k\Tilde{p}}{1-\Tilde{p}}>\frac{\Tilde{p}}{1-\Tilde{p}}>1,$$ then 
\begin{align*}
   R&=\int_{l_a}|dw|=\int_{H(l_a)}\left|\frac{dw}{dz}\right||dz|\\
   &=\int_{H(l_a)}|\eta||dz|=\infty,
\end{align*}
which contradicts with $R<\infty$.

\textbf{Case 2: }$\psi(z_0)= a_j$ for some $j\in \{1,2,\cdots,q\}$ . Then $\nu_{\psi}(z_0)\ge e_j$, and $\nu_{\psi'}(z_0)=\nu_{\psi}(z_0)-1$.  
\begin{align*}
    \nu_{\eta}(z_0)
    &=\frac{\Tilde{p}}{1-\Tilde{p}}(\nu_{\psi}(z_0)(1-\frac{1}{e_j}-\epsilon')-\nu_{\psi}(z_0)+1)\\
    &=\frac{\Tilde{p}}{1-\Tilde{p}}(1-\nu_{\psi}(z_0)(\frac{1}{e_j}+\epsilon'))\\
    &\le \frac{\Tilde{p}}{1-\Tilde{p}}(1-e_j(\frac{1}{e_j}+\epsilon'))\\
    &=-\frac{e_j\epsilon'\Tilde{p}}{1-\Tilde{p}}<-1
\end{align*}
then $\int_{H(l_a)} |\eta||dz|=+\infty$ and forces a contradiction.

\end{proof}

  Observing that
  
  
  $$\frac{dw}{dz}=\frac{g
\Pi_{i=1}^l(\psi-c_i)^{\Tilde{p}(1+p-\epsilon')}
\Pi_{j=1}^{q}\left({\psi-a_j}\right)^
{\Tilde{p}(1-\frac{1}{e_j}-\epsilon')}}
{(\psi')^{\Tilde{p}}}
 \left(\frac{dw}{dz}\right)^{\Tilde{p}},$$
then


$$\left|\frac{dz}{dw}\right|^2=\frac{|(\psi\circ H)'|^{2\Tilde{p}}\left|\frac{dz}{dw}\right|^{2\Tilde{p}}}{|g\circ H|^2\Pi_{i=1}^l|\psi\circ H-c_i|^{2\Tilde{p}(1+p-\epsilon')}
\Pi_{j=1}^{q}|\psi\circ H-a_j|^
{2\Tilde{p}(1-\frac{1}{e_j}-\epsilon')}},$$
and the pull back metric on $\mathbb D (R)$ is 

\begin{align*}
    H^*d\Tilde{s}^2 &= \bigg ( \left(\Pi_{i=1}^l|\psi-c_i|^2\right)(1+|\psi|^2)^{m-l}|g|^2\bigg)\circ H \left|\frac{dz}{dw}\right|^2|dw|^2\\
    &=\frac{\left(\Pi_{i=1}^l|\psi(w)-c_i|^2\right)(1+|\psi(w)|^2)^{m-l}|g(w)|^2|\psi'(w)|^{2\Tilde{p}}}{|g(w)|^2\Pi_{i=1}^l|\psi(w)-c_i|^{2\Tilde{p}(1+p-\epsilon')}\Pi_{j=1}^{q}|\psi(w)-a_j|^{2\Tilde{p}(1-\frac{1}{e_j}-\epsilon')}}|dw|^2\\
&=\left(\frac{(1+|\psi|^2)^{\frac{(m-l)p}{2}}\psi'}{\Pi_{i=1}^{l}(\psi-c_i)^{1-\epsilon'}\Pi_{j=1}^{q}(\psi-a_j)^{1-\frac 1 e_j-\epsilon'}}\right)^{2\Tilde{p}}|dw|^2\\
&\le \left(B\frac{2R}{R^2-|w|^2}\right)^{2\Tilde{p}} .
\end{align*}
Thus

\begin{equation}\label{int}
    \int_{H(l_a)} d\Tilde{s}=\int_{l_a}H^*d\Tilde{s}\le\int_0^R\left(\frac{2R}{R^2-|r|^2}\right)^{\Tilde{p}} dr . 
\end{equation}

The integral (\ref{int}) is convergent since $0<\Tilde{p}<1$, which contradicts to the completeness of $(M, d\Tilde{s}^2 )$.    
\end{proof}

\begin{rem}
 In equation $(\ref{mr})$, 
 \begin{itemize}
    \item If $m=0, |E_f|=1$($M$ is 2-degenerate), then $\gamma \le 2$; 
    \item If $m=1, |E_f|=2$($M$ is 1-degenerate of hyperbolic type), then $\gamma \le 2$;
    \item If $m=1, |E_f|=0$($M$ is 1-degenerate of elliptic type), then $\gamma \le 4$. 
 \end{itemize}
 These conclusions are agree with the Proposition \ref{special}.
\end{rem}
Thus, by combining the above proof process and Proposition \ref{special}, we get  Main Theorem A. 

\section{Unicity problem}
\label{sec4}
 
 Similarly as in Section \ref{sec3}, we first consider  Riemann surface $M$ and  metric 
 $$d\Tilde{s}^2=2(1+|\psi_1|^2)(1+|\psi_2|^2)|dh|^2.$$
 Thus we can easily get the following theorem by Theorem $1.2$ in \cite{MR3784760}. 
 
 \begin{thm}\label{unicity}
 Let $M$ and $\hat{M}$ be two complete non-flat space-like stationary surfaces in $\mathbb R^{3,1}$ with Gauss maps $G=(\psi_1,\psi_2)$ and $\hat{G}=(\hat{\psi}_1, \hat{\psi}_2)$ respectively.  Assume  $\Phi: M\to \hat{M}$ is  a conformal diffeomorphism 
     and $\psi_1,\psi_2, \hat{\psi}_1, \hat{\psi}_2$ are not constant. If for each $i(i=1,2)$, $\psi_i$ and $\hat{\psi}_i$ share $p_i>4$ distinct values and  $\psi_i\ne \hat{\psi}_i\circ\Phi $, then we have $\min\{p_1, p_2\}\le 6.$
In particular, if $p_1\ge 7$ and $p_2\ge 7$, then either $\psi_1\equiv \hat{\psi}_1\circ\Phi $ or $\psi_2\equiv \hat{\psi}_2\circ\Phi $, or both hold.

\begin{rem}
    Since every minimal surface in $\mathbb R^3$ is also a space-like stationary surface in $\mathbb R^{3,1}$ with Gauss map $G=(\psi_1,\psi_2)$ and $\psi_2=-\frac{1}{\psi_1}$, then theorem \ref{uni} and the optimality in $\mathbb R^3$ tell us the above result is also optimal.
\end{rem}

\end{thm}

Next, we will consider the space-like stationary surface with Gauss map $G=(\psi_1,\psi_2)$ satisfying $\psi_2=f(\psi_1)$, where $f$ is a rational function with degree $m$. Now, assume $\Phi: M\to \hat{M}$ be the conformal diffeomorphism, and $(\psi_1, \psi_2)$ and $(\hat{\psi_1}, \hat{\psi_2} )$ be the Gauss map of $M$ and $\hat{M}$ respectively, where $\psi_2=f(\psi_1)$, $\hat{\psi_2}= f(\hat{\psi_1})$. For brevity, we denote $\psi_1, \hat{\psi}_1\circ \Phi$ by $\psi, \hat{\psi}$.

 If  $M, \hat{M}$ are 2-degenerate $(m=0)$ or  1-degenerate of hyperbolic type($m=1,|E_f|=2$), then they are entire graphs over $\mathbb R^2$ and their universal covering spaces  are conformally equivalent to the whole complex plane. Then the Main Theorem B is followed by the Nevanlinna unicity theorem \cite{MR1555233} for meromorphic functions on the complex plane $\mathbb C$.

If $M, \hat{M}$ are 1-degenerate of elliptic type($m=1, |E_f|=0$), then they can be deformed by  minimal surfaces in $\mathbb R^3$, and the Main Theorem B is followed by Theorem \ref{uni}.

Thus, we only need to consider the proof when $m\ge 2$ and $M, \hat{M}$ are 1-degenerate of parabolic type($m=1, |E_f|=1$). For each $\alpha, \beta \in \bar{\mathbb C}$, the chordal distance\cite{MR1167375} between $\alpha, \beta$ is 
$$|\alpha, \beta|=\frac{|\alpha-\beta|}{\sqrt{1+|\alpha|^2}\sqrt{1+|\beta|^2}},$$
if $\alpha\ne \infty, \beta\ne \infty$, and $|\alpha, \beta|=|\beta, \alpha|=\frac{1}{\sqrt{1+|\alpha|^2}}$, if $\beta=\infty$.

\begin{lem}\cite{MR1219880}\label{le}
   Let $f$ and $\hat{f}$ be two mutually distinct non-constant meromorphic functions on a Riemann surface $M$ and $q$ distinct points $a_1, a_2,\cdots, a_q(q>4)$. Assume that $f^{-1}(a_j)=\hat{f}^{-1}(a_j)(1\le j\le q).$ For $a_0>0$ and $\epsilon$ with $q-4>q\epsilon>0$, set

   $$\lambda:=\left(\Pi_{i=1}^q|f, a_i|\log\left(\frac{a_0}{|f, a_i|^2}\right)\right)^{-1+\epsilon}, $$
   $$\hat{\lambda}:=\left(\Pi_{i=1}^q|\hat{f}, a_i|\log\left(\frac{a_0}{|\hat{f}, a_i|^2}\right)\right)^{-1+\epsilon}, $$
   \begin{equation}\label{tau}
d\tau^2:=|f,\hat{f}|^2\lambda\hat{\lambda}\frac{f'}{1+|f|^2}\frac{\hat{f}'}{1+|\hat{f}|^2}
\end{equation}
   outside the $E:=\cup_{i=1}^qf^{-1}(a_i)$ and $d\tau^2=0$ on $E$. Then, for a suitably chosen $a_0$, $d\tau^2$ is continuous on $M$, and has strictly negative curvature on the set $\{d\tau^2\ne 0\}$.
\end{lem}
\begin{lem}\cite{MR1219880}\label{gu}
  Let $f$ and $\hat{f}$ be two mutually distinct non-constant meromorphic functions on a Riemann surface $M$ satisfies  the same assumption as in Lemma \ref{le}, then for the metric $d\tau^2$  defined by (\ref{tau}), there is a constant $C>0$ such that 
  $$d\tau^2\le C\frac{4R^2}{(R^2-z^2)^2}|dz|^2.$$
\end{lem}
Combine the completeness of  $M$ and inequality (\ref{metric}), the conformal metric $ds^2$ 
can be stated as follows:
$$ds^2=\Pi_{i=1}^l
|\psi-c_i|(1+|\psi|^2)^{m-l}|h|^2|dz|^2,$$
where $h$ is a  holomorphic function with no zeros on a simply connected open set $U$. 

Since there exists a local non-zero holomorphic function $\zeta$ on $U$ such that $ds^2=|\zeta|^2\Phi^*(d\hat{s}^2)$, thus 
$$\Pi_{i=1}^l
|\psi-c_i|^2(1+|\psi|^2)^{m-l}|h|^2|dz|^2=|\zeta|^2\Pi_{i=1}^l
|\hat{\psi}-c_i|^2(1+|\hat{\psi}|^2)^{m-l}|\hat{h}|^2|dz|^2,$$
$$\Rightarrow \Pi_{i=1}^l
|\psi-c_i|(1+|\psi|^2)^{\frac{m-l}{2}}|h|=|\zeta|\Pi_{i=1}^l
|\hat{\psi}-c_i|(1+|\hat{\psi}|^2)^{\frac{m-l}{2}}|\hat{h}|.$$
Let $k$ be a non-zero holomorphic function on $U$ satisfies $k^2=h\hat{h}\zeta$, then 
$$ds^2=|k|^2\Pi_{i=1}^l
(|\psi-c_i||\hat{\psi}-c_i|)(1+|\psi|^2)^{\frac{m-l}{2}}(1+|\hat{\psi}|^2)^{\frac{m-l}{2}}|dz|^2.$$
Assume there are $q$ distinct points $a_1,\cdots, a_q\in \bar{\mathbb C}$ such that $\psi^{-1}(a_i)=\hat{\psi}^{-1}(a_i)$ for all $1\le i\le q$. Since $\psi^{-1}(E_f)=\hat{\psi}^{-1}(E_f)=\emptyset$, then $E_f\subset \{a_1,\cdots, a_q\}$. We may assume $a_q=\infty\in E_f$.

Since $q\ge m-|E_f|+6=m-l+5$, then $q-m+l-4>0$. Take a positive real number $\delta$ with
$$\frac{q-m+l-4}{q}>\delta>\min\left\{\frac{q-m+l-4}{m-l+q}, \frac{q-4-2m+2l}{q}\right\}$$
and set 
$$p=\frac{m-l}{q-4-q\delta}<1.$$
Then 
$$\frac{p}{1-p}>1,\quad \frac{\delta p}{1-p}>1.$$

If $\psi\not\equiv\hat{\psi}$, consider the function
\begin{equation}\label{eta}   
\Tilde{\eta}=k^{\frac{1}{1-p}}\frac{\Pi_{i=1}^l
\left((\psi-c_i)(\hat{\psi}-c_i)\right)^{\frac{1}{2(1-p)}}\Pi_{j=1}^{q-1}\left((\psi-a_j)(\hat{\psi}-a_j)\right)^{\frac{p(1-\delta)}{2(1-p)}}}{\left((\psi-\hat{\psi})^2\psi'\hat{\psi}'\Pi_{j=1}^{q-1}(1+|a_j|^2)^{1-\delta}\right)^{\frac{p}{2(1-p)}}}
\end{equation}
on $M'=M\setminus E$, where $E=\{z\in M, \psi'(z)\hat{\psi}'(z)=0, \text{ or } \psi(z)=\hat{\psi}(z)\}.$

Denote 
\begin{equation}\label{f}   
w=\Tilde{F}(z)=\int\Tilde{\eta} dz,\end{equation}
by similar method in the proof of Main Theorem A, we can find a neighborhood $U$ of point $0$ and a  positive number R, such that the inverse map $\Tilde{H}=(\Tilde{F}|_U)^{-1}: \mathbb D(R)\to U\subset M' $ is a holomorphic map. Since 
$d\sigma^2=|\Tilde{\eta}|^2|dz|^2$ has strictly negative  curvature on $M'$  by Lemma \ref{le}, then $R<\infty$.
And there exists a point
$a$ with $|a|=R$, such that for the line segment
$$l_a=\{ta, 0\le t< 1\},$$
the image $\Tilde{\gamma}=\Tilde{H}(l_a)$ is a divergent curve in $M'$ as $t$ tends to $1$.


\begin{lem}
There exists a point  with $|a|=R$ such that $\Tilde{H}(l_a)$ is a divergent curve in $M$.   
\end{lem}
\begin{proof}
 It is suffices  to consider that $\Tilde{H}(l_a)$ tends to a point $z_0\in E$.
 
 \textbf{Case 1}: 
 If $\psi(z_0)=\hat{\psi}(z_0)$ and there exists some $a_j,$ such that $\psi(z_0)=\hat{\psi(z_0)}=a_j$, then 
 \begin{equation*}
\begin{split}
    \nu_{\Tilde{\eta}}(z_0) &=  
    \frac{p}{1-p}\left((1-\delta)(\nu_{\psi}(z_0)+\nu_{\hat{\psi}}(z_0))
    -2\min\{\nu_{\psi}(z_0),\nu_{\hat{\psi}}(z_0)\}+\nu_{\psi}(z_0)+\nu_{\hat{\psi}}(z_0)-2\right)\\
    &\le -\frac{\delta p}{1-p}(\nu_{\psi}(z_0)+\nu_{\hat{\psi}}(z_0))\\
    &<-\frac{\delta p}{1-p}.
\end{split}
\end{equation*}
Otherwise, $$\nu_{\Tilde{\eta}}(z_0)=-\frac{p}{1-p}(2\min\{\nu_{\psi}(z_0),\nu_{\hat{\psi}}(z_0)\})\le -\frac{2\delta p}{1-p}.$$
Since $\frac{\delta p}{1-p}>1$, then
\begin{align*}
   R&=\int_{l_a}|dw|=\int_{\Tilde{\gamma}}\left|\frac{dw}{dz}\right||dz|\\
   &=\int_{\Tilde{\gamma}}|\Tilde{\eta}||dz|=\infty,
\end{align*}
which contradicts with $R<\infty$.

\textbf{Case 2}: $\Tilde{H}(l_a)$ tends to a point $z_0$ such that$\psi'(z_0)\hat{\psi'}(z_0)=0$,  we can easily get $\nu_{\Tilde{\eta}}(z_0)<-\frac{p}{1-p}<-1$. This also contradicts that $R$ is finite.
\end{proof}

Combine with (\ref{eta}) and (\ref{f}), we obtain
$$\left|\frac{dw}{dz}\right|=\frac{|k|\Pi_{i=1}^l
\left(|\psi-c_i||\hat{\psi}-c_i|\right)^{\frac 1 2}\Pi_{j=1}^{q-1}\left(|\psi-a_j||\hat{\psi}-a_j|\right)^{\frac{p(1-\delta)}{2}}}{\left(|\psi-\hat{\psi}|^2|\psi'||\hat{\psi}'|\Pi_{j=1}^{q-1}(1+|a_j|^2)^{1-\delta}\right)^{\frac p 2}}\left|\frac{dw}{dz}\right|^{p}.$$

Set $g(w)=\psi(\Tilde{H}(w)), \hat{g}(w)=\hat{\psi}(\Tilde{H}(w))$, since $g'=\psi'\frac{dz}{dw}, 
 \hat{g}'=\hat{\psi}'\frac{dz}{dw}$, then 
$$\left|\frac{dz}{dw}\right|^2=\frac{\left(|g-\hat{g}|^2|g'||\hat{g}'|\Pi_{j=1}^{q-1}(1+|a_j|^2)^{1-\delta}\right)^{p}}{|k\circ \Tilde{H}|^2\Pi_{i=1}^l
\left(|g-c_i||\hat{g}-c_i|\right)\Pi_{j=1}^{q-1}\left(|g-a_j||\hat{g}-a_j|\right)^{p(1-\delta)}}$$
Therefore, the induced metric of  $\mathbb{D}(R)$ from $M$ by $\Tilde{H}$ is 

\begin{align*}
\Tilde{H}^*ds^2&=|k\circ \Tilde{H}|^2\Pi_{i=1}^l
(|g-c_i||\hat{g}-c_i|)(1+|g|^2)^{\frac{m-l}{2}}(1+|\hat{g}|^2)^{\frac{m-l}{2}}\left|\frac{dz}{dw}\right|^2|dw|^2 \\
&=\left(\frac{(1+|g|^2)^{\frac{m-l}{2p}}(1+|\hat{g}|^2)^{\frac{m-l}{2p}}|g-\hat{g}|^2|g'||\hat{g}'|\Pi_{j=1}^{q-1}(1+|a_j|^2)^{1-\delta}}{\Pi_{j=1}^{q-1}|g-a_j|^{1-\delta}|\hat{g}-a_j|^{1-\delta}}\right)^{p}|dw|^2\\
&=\left(\mu^2\Pi_{i=1}^q(|g,a_i||\hat{g},a_i|)^{\epsilon}\left(\Pi_{i=1}^q\log\frac{a_0}{|g,a_i|^2}\log\frac{a_0}{|\hat{g},a_i|^2}\right)^{1-\epsilon}\right)^p|dw|^2,
\end{align*}
where $\mu$ is the function with $d\tau^2=\mu^2|dw|^2$. On the other hand, for a given $\epsilon$, it holds that 
$$\lim_{x\to 0}x^{\epsilon}\log^{1-\epsilon}\frac{a_0}{x^2}<\infty. $$
Thus there exists a constant $C_1$ such that
$$\Tilde{H}^*ds^2\le C_1\mu^{2p}|dw|^2.$$
By Lemma \ref{gu}, 
$$\int_{\Tilde{\gamma}}ds=\int_{{l_a}}\Tilde{H}^*ds<C\int_{l_a}\left(\frac{R}{R^2-|w|^2}\right)^p|dw|<\infty,$$
which contradicts with the completeness of $M$. We have necessarily $\psi=\hat{\psi}$, and the proof of Main Theorem B is completed.

\begin{rem}
In this context, we cannot uniformly discuss whether the results are optimal or not, because values of $m$ and $|E_f|$ vary. Although for each fixed $m$, the value of $|E_f|$ may also have several situations.
\end{rem}

\backmatter

\bibliography{bibouli}
\end{document}